\documentclass[a4paper,fleqn,11pt]{amsart}
\usepackage{natbib}
\usepackage{hyperref}
\hypersetup{
	pdftitle={L\'evy Driven Polling Systems},
	pdfsubject={Probability Theory},
	pdfauthor={Kamil Kosi\'nski},
	pdfdisplaydoctitle=true,
}

\newcommand{\ee}{\mathbb E}

\newcommand{\rr}{\mathbb R}

\newcommand{\cf}{\mathcal F}

\newcommand{\cl}{\mathcal L}
\newcommand{\cs}{\mathcal S}

\newcommand{\ba}{\boldsymbol a}
\newcommand{\be}{\boldsymbol e}
\newcommand{\bm}{\boldsymbol m}

\newcommand{\bu}{\boldsymbol u}
\newcommand{\bv}{\boldsymbol v}
\newcommand{\bw}{\boldsymbol w}
\newcommand{\bx}{\boldsymbol x}

\newcommand{\0}{\boldsymbol 0}
\newcommand{\1}{\boldsymbol 1}

\newcommand{\bA}{\boldsymbol{A}}
\newcommand{\bE}{\boldsymbol E}
\newcommand{\bF}{\boldsymbol F}
\newcommand{\bG}{\boldsymbol G}
\newcommand{\bL}{\boldsymbol L}
\newcommand{\bR}{\boldsymbol R}

\newcommand{\bH}{\boldsymbol H}
\newcommand{\bW}{\boldsymbol W}
\newcommand{\bX}{\boldsymbol{X}}
\newcommand{\bZ}{\boldsymbol B}

\newcommand{\bvarphi}{\boldsymbol\kappa}

\newcommand{\tG}{\tilde G}
\newcommand{\tS}{\tilde S}

\newcommand{\tX}{\tilde X}

\newcommand{\de}{=_{\rm d}}

\newtheorem{theorem}{Theorem}
\newtheorem{lem}{Lemma}

\newtheorem{cor}{Corollary}

\newtheorem{prop}{Proposition}
\newtheorem{prp}{Property}
\newtheorem{assumption}{Assumption}
\theoremstyle{remark}
\newtheorem{rem}{Remark}

\renewenvironment{proof}[1][\proofname]{\par \normalfont \trivlist
 \item[\hskip\labelsep\itshape #1]\ignorespaces
}{%
 \hspace*{\fill}$\Box$ \endtrivlist
}
\renewcommand{\proofname}{\noindent {\bf Proof}}

\begin{document}
\title[L\'{e}vy-driven polling systems]{L\'{e}vy-driven polling systems and continuous-state branching processes}
\bibliographystyle{plainnat}
\setcitestyle{numbers}

\author{O.J.\ Boxma}
\address{E{\sc urandom} and Department of Mathematics and Computer Science,
Eindhoven University of Technology, the Netherlands}
\curraddr{}
\email{boxma@win.tue.nl}
\thanks{}

\author{J.\ Ivanovs}
\address{E{\sc urandom}, Eindhoven University of Technology, the Netherlands; Korteweg-de Vries
Institute for Mathematics, University of Amsterdam, the
Netherlands} \curraddr{} \email{ivanovs@eurandom.tue.nl}
\thanks{}

\author{K.M.\ Kosi\'nski}
\address{E{\sc urandom}, Eindhoven University of Technology, the Netherlands; Korteweg-de Vries
Institute for Mathematics, University of Amsterdam, the
Netherlands} \curraddr{} \email{kosinski@eurandom.tue.nl}
\thanks{The third author was supported by NWO grant 613.000.701.}

\author{M.\ Mandjes}
\address{Korteweg-de Vries Institute for Mathematics,
University of Amsterdam, the Netherlands;
E{\sc urandom}, Eindhoven University of Technology, the Netherlands; CWI, Amsterdam, the Netherlands}
\curraddr{}
\email{m.r.h.mandjes@uva.nl}
\thanks{}

\date{December 30, 2013}


\begin{abstract}
In this paper we consider a ring of $N\ge 1$ queues served by a
single server in a cyclic order. After having served a queue
(according to a service discipline that may vary from queue to
queue), there is a switch-over period and then the server serves
the next queue and so forth. This model is known in the literature
as a \textit{polling model}.
\par Each of the queues is fed by a non-decreasing L\'evy process, which can
be different during each of the consecutive periods within the
server's cycle. The $N$-dimensional L\'evy processes obtained in
this fashion are described by their (joint) Laplace exponent, thus
allowing for non-independent input streams. For such a system we
derive the steady-state distribution of the joint workload at
embedded epochs, i.e. polling and switching instants. Using the Kella-Whitt martingale, 
we also derive the steady-state distribution at an arbitrary epoch.
\par
Our analysis heavily relies on establishing a link between fluid
(L\'evy input) polling systems and multi-type Ji\v{r}ina processes
(continuous-state discrete-time branching processes). This is done
by properly defining the notion of the \textit{branching property}
for a discipline, which can be traced back to Fuhrmann and Resing.
This definition is broad enough to contain the most important
service disciplines, like exhaustive and gated.
\end{abstract}

\maketitle

\section{Introduction}
\label{INTRO}
Consider a queueing model consisting of multiple queues attended
by a single server, visiting the queues one at a time in a cyclic
order. Moving from one queue to another, the server incurs a
non-negligible switch-over time. Such single-server multiple-queue
models are commonly referred to as {\it polling models}.
Stimulated by a wide variety of applications, polling models have
been extensively studied in the literature, see
\cite{Takagi2,Takagi5,Wisnia} for a series of comprehensive
surveys and \cite{Levy2,Takagi3} for extensive overviews of the
applicability of polling models.

Throughout the vast polling literature, it is almost always
assumed that customers arrive at the queues according to
independent Poisson processes, where, in addition, the service
requirements brought along by these customers are i.i.d.\
sequences; the resulting input processes in the queues thus
constitute independent compound Poisson processes (CPPs).
Correlated arrivals in polling models have received little
attention; see \citet{Levy1} for a treatment of
polling models with correlated CPP input. Classical analysis of
polling systems heavily focuses on keeping track of customers in the
system at embedded epochs, i.e., instants of specific changes in
the system, like {\it polling instants} or {\it switching
instants}. 

A key feature of polling models is the {\it service
discipline}. A service discipline specifies the rule that
determines how long a server will visit a queue (and process any
workload found there). The most important and well known
disciplines include the {\it exhaustive} discipline, {\it gated}
discipline and {\it 1-limited} discipline. Under the exhaustive
discipline, the server will stay at the queue until this queue has
become empty. Under the gated discipline, the server serves
exactly the customers (or: the amount of work) present upon the
beginning of the visit. Under the 1-limited discipline, the server
serves only one customer -- if there is one. The `system's service
discipline' can be any `mixture' of the individual disciplines;
for instance: some of the queues are served according to a gated
discipline, whereas others are served exhaustively.

In \citet{Resing} (see also \citet{Fuhrmann2}) it is
shown that for a large class of classical polling models,
including those with exhaustive and gated service at all queues
(but not 1-limited), the evolution of the system at successive
polling instants at a fixed queue can be described as a multi-type
{\em branching process} (MTBP) with immigration. Models that
satisfy this MTBP-structure allow for an exact analysis, whereas
models that violate the MTBP-structure are considerably more
intricate, and therefore usually intractable. It turns out that it
is exactly the nature of the service disciplines which determines
the MTBP structure of the system. The structure is preserved if
each service discipline satisfies a special property called the
{\it branching property}, see e.g. Property 1 of \cite{Resing}.
The exhaustive and gated disciplines do satisfy this property
whereas the 1-limited does not. A key result for polling models
with the MTBP structure is the joint steady-state distribution of
the queue length (i.e., in terms of the number of customers) at
polling or switching instants of a particular queue.

In this paper we generalize the classical assumptions in several
ways. We consider polling models with L\'evy-driven, possibly
correlated, input streams. More specifically, we assume that the
input process $\bW$ is an $N$-dimensional L\'evy subordinator,
where $N\ge 1$ corresponds to the number of queues, and where
`subordinator' means that the corresponding sample paths are
non-decreasing (in all $N$ coordinates). We refer to this model as
a {\it L\'evy-driven polling model} with input process $\bW$. If
the queue under consideration, say queue $i$, is not in service,
its workload evolves according to the subordinator $W_i(t)$,
whereas during its service time it is described by $W_i(t)$ minus
drift $t$. It is important to note that, in fact, this paper
considers a slightly richer class of models, in which the workload
level while in service behaves as a spectrally positive L\'evy
process $A_i$ with negative drift (that decreases on average); 
here `spectrally positive' means that the underlying process has 
positive jumps only. We
remark that the class of spectrally positive L\'evy processes with
negative drift is used frequently in the theory of {\it storage
processes} to model the storage level (workload) of queues, dams
or fluid models, see e.g. \citet{kyprianou}, and \citet{Prabhu}
for an early reference.

We recall that L\'evy processes are processes with stationary,
independent increments;  it is stressed, however, that the
components of the $N$-dimensional L\'evy process need not
necessarily be independent. The class of L\'evy processes is rich
and covers Brownian motion, linear increment processes and CPPs as
special cases. The generalization from CPPs to L\'evy input
implies that we can no longer speak of notions such as customers
and queue lengths; this explains why we focus on the (joint) {\em
workload} process. While quite a few studies have been devoted to
a single server single queue model with L\'evy input (see, e.g.,
\citet[Chapter 4]{Prabhu} and \citet[Chapter 14]{Asmussen}), 
there is hardly any literature on L\'evy-driven
polling systems. An exception is \citet{Eliazar}, who
considers such systems only for the gated discipline, independent input processes and does not allow for
spectrally positive L\'evy processes. His analysis follows a dynamical-systems approach:
a stochastic Poincar\'e map, governing the one-cycle
dynamics of the polling system is introduced, and its statistical characteristics are studied. 
This approach differs from ours; we identify a branching structure in L\'evy-driven polling models
as will be explained later.
By considering the input as an $N$-dimensional L\'evy process $\bW$ instead of $N$
one-dimensional processes $W_i$, we accomplish an easy
incorporation of {\it correlation} between the inputs to different
queues. This is due to the fact that every L\'evy process is
uniquely characterized by its characteristic exponent, which in
the multidimensional case also captures the correlation structure
between the individual components.

Considering polling models with L\'evy input opens several new
perspectives. Firstly, the theory of L\'evy processes was strongly
developed in recent years, and its application appears to lead to
more simplified derivations of many results which, for the case of
compound Poisson input, are only obtained after detailed
calculations. Secondly, having L\'evy input leads to significant
generalizations of known results. Such generalizations are
theoretically interesting, but also, owing to the inherent
flexibility of L\'evy processes, offer various new possibilities from the viewpoint of
applications.
Polling models have found applications in many different areas,
like (i) Maintenance (a patrolling repairman); (ii)
Stochastic Economic Lotsizing (a machine producing products of various types upon demand);
(iii) road traffic (traffic lights at signalized intersections);
and (iv) protocols in computer and communication networks
(Bluetooth;
token ring protocols; protocols for web servers and routers).
Almost invariably, it has been assumed in the polling literature that the input process is composed
of a number of independent compound Poisson processes.
We allow L\'evy input, and correlation between the various input streams, and
different input processes during different visit and switch-over periods.
This gives much additional modelling capability.
E.g., in stochastic economic lotsizing  it is quite natural to have correlations between
the arrival processes of demands of different product types.
And in road traffic as well as in communications,
it is sometimes better to model traffic as a fluid than as separate customers;
indeed, a special case of our model is the
situation in which there is a constant fluid input in one queue of
the polling model, and a compound Poisson input in an other queue.
As another example, while a Brownian motion component may
not be natural in representing work inflow, it may represent
realistic fluctuations in the speed of the server. Recall that a
served queue is modelled by a spectrally positive L\'evy process
with negative drift allowing for incorporation of a Brownian
component.

The transition from CPPs to L\'evy subordinators deprives us from
the possibility of using the branching property from
\citet{Resing}, which is stated in terms of customers (which are of
a discrete nature) in the system, and therefore has no simple
translation to our continuous state-space setting. In our paper we
identify the analogous property, also referred to as the branching
property, for the disciplines in the L\'evy framework, that
enables to identify a branching structure in our system. This
allows us to mimic Resing's approach, and to describe the
multidimensional workload in the system at successive polling
instants at a fixed queue as a multi-type continuous state-space
(discrete time) branching process. This branching process is
referred to in the sequel as multi-type Ji\v{r}ina branching
process (MTJBP) due to \citet{Jirina}, who introduced
the notion of continuous state-space branching processes and paid
special attention to discrete-time processes (called {\it
Ji\v{r}ina processes} in the literature). The relation between
L\'evy-driven polling models and continuous state-space branching processes 
has been observed before by \citet{altman} in a special case
strongly relying on the assumption imposed that all the
queues are fed by {\em identical} L\'{e}vy subordinators. 
This relation was only used to derive the first two
waiting-time moments and did not focus on the underlying structure of the branching
process. The observation that L\'evy-driven polling models can be completely
characterized (in terms of distribution) by a continuous-state space branching process
of the MTJBP form is therefore novel.

Another performance measure which is analyzed in the paper, again
for disciplines satisfying our new branching property, is the
Laplace-Stieltjes Transform (LST) of the steady-state distribution
of the joint workload in the queues at an {\em arbitrary} epoch.
The classical polling literature focuses strongly on joint queue
lengths at polling epochs, and contains results for {\em marginal}
queue lengths and workloads at arbitrary epochs, but we are not
aware of any general results for {\em joint} queue-length or
workload distributions at arbitrary epochs -- with the exception
of the recent paper by \citet{Czerniak} that
considers constant fluid input at all queues for a very special
case of our model. We employ the Kella-Whitt martingale \cite{Kella}
to obtain this result. A similar approach
has been used before; e.g., \citet{Boxma}
give the steady-state storage level transform for a L\'evy-driven
queueing model with service vacations.

\subsection*{Contribution}

This paper casts a broad class of queueing models into a single
general framework. More specifically the contributions are the
following. First, we consider general, L\'evy-driven polling
models instead of the classical models with CPP inflow. Second, we
let the input $\bW$ change at polling and switching instants,
whereas in classical polling models the input processes are
typically fixed once and for all. Third, we allow for correlation
between the individual input processes (correlated arrivals received little attention so far,
see \cite{Levy1}). Fourth, we introduce a new
class of service disciplines satisfying a novel branching
property, and we relate L\'evy-driven polling models to MTJBP.
This class is broad and contains the well known exhaustive and
gated disciplines. Fifth, we provide the LST of the joint
steady-state workload distribution at an arbitrary epoch, which is
a new result even for classical polling models. Finally, we show
that the stability of our system does not depend on the
disciplines used at different queues, and can be formulated in
terms of rates of input (which leads to an intuitively appealing
criterion).

\subsection*{Organization of the paper}

The remainder of this paper is organized as follows. In
\autoref{Model description} we describe the model and the service
disciplines that are considered in this paper. \autoref{MTJBP}
presents a brief introduction on MTJBPs, and some additional
intuition behind these processes. We also state a limit theorem
for MTJBPs with immigration. \autoref{Polling and Jirina}
contains one of the two main theorems in this paper. It is shown
that in our model the workload level at different queues at
successive epochs that the server reaches a fixed queue is an
MTJBP with immigration. This leads to an expression for the LST of
the stationary joint workload distribution at different queues at
these epochs. \autoref{Steady-state distribution} contains the
second main theorem of this paper. We derive the LST of the
stationary joint workload distribution at an arbitrary epoch.
In \autoref{VIP} we show that our  results carry over in a
straightforward way to the situation in which $\bW$ changes
between polling and switching instants. In 
\autoref{Ergodicity} we present a discussion of the ergodicity of the
most general model, i.e., the model addressed in 
\autoref{VIP}. We conclude in \autoref{DACR} by suggesting
possible further generalizations.

\subsection*{Notation}

In the sequel, for any random variable $X$ we denote its LST by
$\tX$, i.e., $\tX(u)=\ee e^{-uX}$ for $u\in\rr$ such that $\tX$
exists. Throughout this paper, we will only use {\it spectrally
positive} L\'evy processes, that is, processes which are allowed
to have positive jumps only (and therefore containing the class of
all aforementioned L\'evy subordinators). Recall from \autoref{INTRO}
that the class of spectrally positive L\'evy processes is used frequently in the theory of storage
processes, see, e.g., \cite{kyprianou, Prabhu}.
For a background on
L\'evy processes see e.g.\ \cite{kyprianou,Sato}. Such a process
$\{L(t),t\ge 0\}$ can be uniquely characterized by its Laplace
exponent: a function $\phi:\rr_+\rightarrow\rr$ say, such that
$\ee e^{-u L(t)}= e^{-t\phi(u)}$.

To explicitly distinguish the multidimensional case from the
single-dimensional case, we will use bold symbols to denote
vectors and plain symbols to denote their coordinates, so that
$\bu\equiv(u_1,\ldots,u_d)$ for some $d>1$. The inner product of
two vectors $\bu$ and $\bv$ will be denoted by $\bu\cdot\bv$.
Finally, for a random vector $\bX$ and multidimensional spectrally
positive L\'evy process $\{\bL(t),t\ge 0\}$ we define its LST and
Laplace exponent analogously to the corresponding one-dimensional
objects. All inequalities for vectors should be understood
coordinate-wise. For sequences of one-dimensional objects we will
use subscripts and write them as $(a_n)$, whereas in the
multidimensional case we will sometimes use superscripts and write
them as $(\ba^n)$ to avoid double (or sometimes even triple)
subscripts.

As mentioned above, we focus on the system's workload process, and
would like to show that it possesses a specific branching
structure. To this end, we talk about {\it children} of a workload
portion $x$. Because we can treat a portion of workload $x$ as any
number $n$ of smaller portions $x/n$, we need to be able to talk
about infinitesimally small portions of workload. That is why we
shall adopt the language of fluid queues, and an infinitesimally
small portion of workload can be regarded as a {\it drop}. Hence,
from now on we shall use the terms workload and amount of fluid
interchangeably.

\section{Model description}
\label{Model description}

We consider a system of $N$ infinite-buffer fluid queues,
$Q_1,\ldots, Q_N$, and a single server. The server moves along the
queues in a cyclic order. When leaving $Q_j$ and before moving to
$Q_{j+1}$ (where, by convention, $Q_{N+1}$ should be understood as
$Q_1$), the server incurs a switch-over period whose duration is a
positive random variable $S_j$ independent of anything else. Queues are fed
by an $N$-dimensional L\'evy subordinator $\bW=\{\bW(t),t\ge0\}$
with Laplace exponent $\phi$. The server's work at $Q_i$ is
modelled by a spectrally positive L\'evy process $A_i$ with negative
drift (i.e., $\ee A_i(1)<0$), so that the work in the system evolves (while the
server is at $Q_i$) according to a L\'evy process
\[
\bA_i(t)
:=\Big(W_1(t),\ldots,W_{i-1}(t),A_i(t),W_{i+1}(t),\ldots,W_{N}(t)\Big).
\]
The Laplace exponent of $\bA_i(t)$ is denoted through
$\phi^{\bA}_i(\bu)$. 
\begin{rem}
Taking a spectrally positive L\'evy process $A_i$ rather than just a L\'evy subordinator $W_i$
allows for usage of a slightly bigger class of input processes.
For instance one can have a Brownian component as a component 
in an input process. The use of reflected Brownian motion is quite common in queueing theory, 
as it is the limiting model for a wide class of queueing models under the functional central limit theorem,
see, e.g., \citet{Whitt}.  
\end{rem}
\begin{rem}
In \autoref{Model description}, \autoref{Polling and Jirina} and 
\autoref{Steady-state distribution} the input process
$\bW$ remains fixed. In \autoref{VIP} we show that one can still
analyze the joint workload process if $\bW$ is allowed to change
at polling and switching instants. In classical polling models this
could correspond to, e.g., having arrival rate $\lambda_{ij}$ at $Q_i$
when the server is at $Q_j$.
\end{rem}
The service disciplines that we consider in
this paper satisfy the following property.
\begin{prp}
\label{Property2} If the server arrives at $Q_i$ to find the
workload level $x$ there, then during the course of the server's
visit, this workload is replaced by $\bH_i(x)$, where
$\{\bH_i(x),x\ge 0\}$ is an $N$-dimensional L\'evy subordinator
with Laplace exponent $\eta_i$, which can be any Laplace exponent
corresponding to an $N$-dimensional subordinator.
\end{prp}
In other words, if the server finds workload vector $\bx$ at the
time of arrival at $Q_i$ then the workload vector at the end of
the service of this queue becomes $\bx-x_i\be_i+\bH_i(x_i)$. Note
that any replacement process should stay positive so that work
does not become negative in $Q_i$ and does not decrease in the
other queues. It is also obvious that such a process should be
increasing in $x$. Therefore the assumption that $\bH_i(x)$ is a
subordinator is intuitively clear and natural. Moreover, due to
the independent stationary increments property of any L\'evy
process, we have that $\bH_i(x+y)\de\bar{\bH_i}(x)+
\bar{\bH_i}(y)$, where $\bar{\bH_i}(x)$ and $\bar{\bH_i}(y)$ are
independent with the same distribution as $\bH_i(x)$ and
$\bH_i(y)$, respectively (and `$\de$' denotes equality in
distribution). Note that this properties essentially says  that
each {\it drop} of the fluid in the served queue is treated in an
i.i.d.\ manner. It is further observed that 
\autoref{Property2} is a continuous analogue of the branching property
from \citet{Fuhrmann1}. Note that we allow different
service disciplines at different queues, as long as they obey
\autoref{Property2}.

\subsection*{Examples}
It is readily verified that the important exhaustive and gated
disciplines both satisfy \autoref{Property2}. 
\subsection*{The gated discipline}
Under the gated
discipline, the server only serves the workload that was present at the
start of the visit. Fluid flowing into the queue during the course
of the visit is served in the next visit. Assuming (c.f., \autoref{rem1} below) 
that the server
works with rate 1, i.e. $A_i(t)=W_i(t)-t$, and finds an amount of
work $x$ upon arrival, the time $\tau_i(x)$ spent in $Q_i$ is simply
$\tau_i(x)=x$, so that
\[
\bH_i(x)=\bW(\tau_i(x))=\bW(x)
\]
and
\begin{equation}\label{gated}
\eta_i(\bu)=\phi(\bu).
\end{equation}
\begin{rem}
\label{rem1}
Observe that in the case of the gated discipline we assumed that
$A_i(t)=W_i(t)-t$, so the workload level in $Q_i$, during the
server's visit, behaves as $x+W_i(t)-t$, where $x$ is the starting
level. This is a standard assumption in the theory of {\it storage
processes}, although, in principle, $A_i$ could be any spectrally
positive L\'evy process with negative drift. Such processes are
used frequently to model the workload level in fluid queues. The
gated discipline, however, becomes ill-defined in such a general
setting, thus every time we speak of the gated discipline we
tacitly assume that the server works with rate 1. 
\end{rem}
\subsection*{The exhaustive discipline}
Verification of the validity of \autoref{Property2} for
exhaustive service is somewhat more involved. To this end, first
recall that the server continues to work until the queue becomes
empty. Fluid arriving during the course of the visit is served in
the current visit. Let $T_i(x):=\inf\{t\ge 0: A_i(t)=-x\}$ be the
time needed to empty the queue with initial workload level $x$. It
is known that for $A_i(t)$ a spectrally positive L\'evy process
with negative drift, $T_i(x)$ is an a.s.\ finite stopping time.
Using the property of stationary and independent increments
generalized to stopping times~\cite[Theorem 3.1]{kyprianou} one
can see that $\bH_i(x)$ is a L\'{e}vy process. We need to compute
\[\ee \exp(-\bu\cdot\bH_i(x))=\ee \exp\left(-\sum_{j\neq
i}u_jW_j(T_i(x))\right),\] which does not depend on $u_i$. Now
consider, for a fixed vector $\bu\geq \0$, the function
$\phi_{i,\bu}(\theta):=\phi^{\bA}_i(u_1,\ldots,
u_{i-1},\theta,u_{i+1},\ldots,u_N)$, where $\theta\geq 0$. It is
easy to see that $\phi_{i,\bu}(0)\geq 0$, because $W_j(t),j\neq i$
are subordinators. Moreover,
$\lim_{\theta\rightarrow\infty}\phi_{i,\bu}(\theta)=-\infty$,
because $A_i(t)$ is not a subordinator. It follows from the
continuity of $\phi_{i,\bu}(\theta)$ that there exists a number
$\psi_i(\bu)\geq 0$ such that $\phi_{i,\bu}(\psi_i(\bu))=0$. Then
consider Wald's martingale
$\exp({-\bu\cdot\bA_i(t)+\phi^{\bA}_i(\bu)t})$ and pick
$u_i=\psi_i(\bu)$. Application of the optional sampling theorem to
the a.s.\ finite stopping time $T_i(x)$ gives $\ee
\exp({-\sum_{j\neq i}u_jW_j(T_i(x))+\psi_i(\bu) x})=1$. Hence $\ee
e^{-\bu\bH_i(x)}=e^{-\psi_i(\bu)x}$ and
\begin{equation}
\label{ex}\eta_i(\bu)=\psi_i(\bu).
\end{equation}
\subsection*{Mixed disciplines}
For a fixed $p\in[0,1]$ one can require that the fraction $p$ of
the present workload (upon arrival) is handled according to
some branching discipline $\bH^{(1)}$ and
the rest according to a different branching discipline $\bH^{(2)}$. Then,
\[
\bH_i(x)\de \bH^{(1)}(px)+\bH^{(2)}((1-p)x)
\]
is a branching discipline. For instance, one can consider a {\it $p$-exhaustive} discipline,
where it is required that the initial amount of workload $x$ be reduced to the level $px$.
Then,
\[
\bH_i(x)\de px + \bH^{\text{exhaustive}}((1-p)x)
\]
is a branching discipline with Laplace exponent
\[
\eta_i(\bu)=pu_i+\psi_i(\bu)(1-p).
\]
\subsection*{Composition of disciplines}
Because the composition of two L\'evy processes is again a L\'evy process, composition
of two branching disciplines is again a branching discipline. In other words, one can
consider a discipline, when upon finishing its service with initial workload level
$x$ according to a branching discipline $\bH^{(1)}$ (with Laplace
exponent $\eta^{(1)}$), the server immediately starts to work again
with initial workload level $\bH^{(1)}_i(x)$ according to a branching discipline $\bH^{(2)}$
(with $\eta^{(2)}$).
That is,
\[
\bH_i(x)=\bH^{(2)}\left(\bH^{(1)}(x)\right)
\]
is a branching discipline with Laplace exponent
\[
\eta_i(\bu)=\eta^{(2)}\left(\eta^{(1)}(\bu)\right).
\]
Naturally, one can think of the composition of more than two branching disciplines.
\subsection*{General method}
Recall that $T_i(x):=\inf\{t\ge 0: A_i(t)=-x\}$.
In the above examples we saw that as soon as we know the time $\tau_i(x)\le T_i(x)$
that the server spends serving $Q_i$ if it finds the workload
level $x$ upon its arrival, the replacement process $\bH_i(x)$ can
be written as $x\be_i+\bA_i(\tau_i(x))$. This relation will be
further exploited in \autoref{Steady-state distribution}.

\section{Multi-type Ji\v{r}ina branching processes}
\label{MTJBP}
The observation that the theory of branching processes of
Bienaym\'e-Galton-Watson type can be extended to random variables
taking their values in a continuous state-space appears to be due
to \citet{Jirina}, who points out that the key to such
an extension is to make the {\it offspring distribution}
infinitely divisible. The effect of an initial quantity of {\it
parent}, may then be described by raising the Laplace exponent of
the number of offspring per unit parent to the appropriate
non-negative power; note the similarity with the discrete case,
where the generating function of the distribution of offspring per
individual is raised to a non-negative integral power.

More formally, let $\{R_{i,j},i,j\ge0\}$ be a sequence of
integer-valued independent random variables, all distributed as
$R$, and let $\{G,G_j, j\ge1\}$ be a sequence of i.i.d.\
integer-valued random variables. Then the Bienaym\'e-Galton-Watson
process with immigration is defined as
\[
X_{n+1}=\sum_{j=1}^{X_n}R_{n+1,j}+ G_{n+1},\quad n=0,1,\ldots,
\]
where $X_0$ is a given (integer-valued) starting random variable.
In the terms of LST's, we have
\[
\ee \left(z^{X_{n+1}}|\cf_n^X\right)=\left(\ee z^R\right)^{X_n}\ee
z^{G},\quad |z|\leq 1.
\]

Now consider an i.i.d.\ sequence of L\'evy subordinators
$\{R_n(x),x\ge0\,,n=1,2,\ldots\}$ characterized by a common
Laplace exponent $\kappa$ and independent from $\{G,G_j, j\ge1\}$.
Then the sequence $X_n$, where
\begin{equation}\label{def:d0}
X_{n+1}=R_{n+1}(X_n)+G_{n+1},\quad n=0,1,\ldots,
\end{equation}
 is called {\it Ji\v{r}ina branching process} with
immigration. Note that every part of $X_n$ reproduces in an
i.i.d.\ way, because $R_{n+1}$ is a L\'{e}vy subordinator. The
distribution corresponding to $\kappa$ may be interpreted as
describing the quantity of offspring per unit quantity of parent.
Models of this type have been analyzed in various papers, see e.g.
\cite{Pakes,Seneta1,Seneta2}. Finally, observe that from
~\eqref{def:d0} it follows that
\begin{equation}\label{def:d1}
 \ee \left(e^{-u X_{n+1}}|\cf^X_n\right)=e^{-X_n\kappa(u)}\tG(u),\quad
 u\ge0.
\end{equation} This relation expresses what has been
said in the first paragraph of this section.

In a natural way this concept extends to the {\it multi-type}
case, where each type reproduces independently from others and
gives rise to a multi-type population. For $i=1,\ldots,N$, let
$\{\bR_i,\bR_i^n(x),x\ge0,n\ge0\}$ be mutually independent
sequences of i.i.d.\ $N$-dimensional L\'{e}vy subordinators with
Laplace exponent $\kappa_i$, and let $\{\bG,\bG^n,n\ge1\}$ be a
sequence of non-negative $N$-dimensional i.i.d.\ random vectors.
We define the one-step evolution of the process $\bX$ through
\begin{equation}\label{def:mult}\bX^{n+1}=\sum_{i=1}^N\bR_i^{n+1}(X^n_i)+\bG^{n+1},\end{equation}
where $\bR^{i,n+1}$ and $\bG^{n+1}$ are assumed to be independent
of $\cf^{\bX}_n$ and $\bX^0$ is a given starting random vector.
 Equation \eqref{def:d1}
then becomes
\begin{equation}
\label{def:dN}
 \ee \left(e^{-\bu\cdot\bX^{n+1}}|\cf^{\bX}_n\right)= e^{-\bX^n\cdot
 \bvarphi(\bu)}\tilde{\bG}(\bu).
\end{equation}
Such a sequence will be called a {\it multi-type Ji\v{r}ina
branching process} (MTJBP) with {\it branching mechanism}
$\bvarphi$ and immigration $\bG$.

Let $m_{i,j}$ be the expected quantity of type $j$ offspring per
unit quantity of parent population of type $i$, i.e.
\[
  m_{i,j}=\ee R_{i,j}(1)=\frac{\partial \kappa_i}{\partial u_j}(0).
\]
An essential role is played by what we will call the {\it mean
matrix} $M\equiv(m_{i,j})_{i,j=1,\ldots,N}$. Note that $M$ is a
non-negative matrix, so by the Perron-Frobenius theory the
spectral radius $\rho_M$ of $M$ is an eigenvalue such that any
other eigenvalue is strictly smaller in absolute value.

 Define $\bvarphi^{(i)}(\bu)$ inductively by
$\bvarphi^{(0)}(\bu)=\bu$ and
$\bvarphi^{(i)}(\bu)=\bvarphi^{(i-1)}(\bvarphi(\bu))$, for
$i=1,2,\ldots$. Finally, let $\|\cdot\|$ denote any norm on
$\rr^N$.
\begin{theorem}
\label{thm:branching} Let $\|\mathbf G\|$ be integrable then the
following holds
\begin{itemize}
  \item if $\rho_M<1$ (subcritical case) then $\bX^n$ converges in distribution
  to a random vector $\bX^\infty\in\rr_+^N$ satisfying
    \begin{equation}\label{eq:br}
  \ee e^{-\bu\cdot\bX^\infty}=\prod_{k=0}^\infty \tilde\bG(\bvarphi^{(k)}(\bu)),\quad\bu\geq\0;
  \end{equation}
  \item if $\rho_M>1$ (supercritical case) and $G_i>0$ with positive probability for
  all $i$ then $\|\bX^n\|\rightarrow\infty$ a.s.\ as $n\rightarrow\infty$.
\end{itemize}
\end{theorem}
We remark that the case of $\rho_M=1$ is substantially more
involved and its treatment is beyond the scope of this paper.
\begin{proof}
The first claim follows from \cite{Altman02}. We prove
the second claim. Let $\bv$ be a non-negative eigenvector
associated to $\rho_M>1$; such a vector exists according to the
Perron-Frobenius theory. Note that $\kappa_i(\theta\bv)$ as a
function of $\theta$, is the Laplace exponent of the L\'{e}vy
process $\bv\cdot\bR_i(t)$, and hence it is concave with
derivative at 0 equal to $(m_{i,1},\ldots,m_{i,N})\cdot\bv=\rho_M
v_i>v_i$. Hence we can choose $\theta^*_i>0$ such that
$\kappa_i(\theta\bv)\geq \theta v_i$ for all
$\theta\in[0,\theta^*_i]$. Let $\theta^*>0$ be the minimum over
$\theta^*_i$, then
\begin{equation}\label{eq:supercritical}\bvarphi(\theta^*\bv)\geq\theta^*\bv.\end{equation}
Combining~\eqref{def:dN} and~\eqref{eq:supercritical} we get
\[0\leq \ee e^{-\theta^*\bv\cdot\bX^{n+1}}\leq\ee e^{-\theta^*\bv\cdot\bX^{n}}\tilde\bG(\theta^*\bv)\leq
(\tilde\bG(\theta^*\bv))^{n+1}\rightarrow 0\] as
$n\rightarrow\infty$, because $\ee e^{-\theta^*\bv\cdot\bG}<1$.
Hence $\theta^*\bv\cdot\bX^n\rightarrow \infty$ and so also
$\|\bX^n\|\to\infty$ with probability 1.
\end{proof}
It is easy to see from the final part of the proof that the
additional condition in the supercritical case, namely $G_i>0$
with positive probability for all $i$, can be substituted with the
following one. It is enough to assume that $\|\bG\|$ is not
identically 0 and there exists a positive vector $\bv$ such that
$M\bv>\bv$. It is known that such a vector exists if $M$ is
irreducible. We do not assume irreducibility of $M$, because our
polling system with exhaustive discipline at the $N$-th queue will
correspond to a matrix $M$ which is not irreducible. We elaborate
more on the stability issue in \autoref{Ergodicity}.

\section{Polling systems and multi-type continuous-state branching processes}
\label{Polling and Jirina}

In this section we shall prove that for our model the amounts of
fluid in the $N$ queues on successive epochs that the server
reaches $Q_1$ form an MTJBP with immigration, which is one of the
main results of this paper. Define $t_n$ as the (random) time
point that the server reaches $Q_1$ for the $n$th time. Let $t_0$
correspond to time 0.
\begin{theorem}\label{theorem1}
 Consider a polling system from \autoref{Model description} with switch-over times
$S_i$. Assume that the service discipline at $Q_i$ satisfies
\autoref{Property2} with Laplace exponent $\eta_i$,
$i=1,\ldots,N$. Then the amount of fluid $\bZ^n$ in the different
queues at time points $t_n$ constitutes a multi-type Ji\v{r}ina
branching process with immigration, where the branching mechanism
$\bvarphi$ is given by the recursive equations
\begin{equation}
\label{recursive_eq}
 \kappa_i(\bu)=\eta_i\left(u_1,\ldots,u_i,\kappa_{i+1}(\bu),\ldots,\kappa_N(\bu)\right),\quad i=1,\ldots,N,
\end{equation}
and the immigration LST $\tilde\bG(\bu)$ is given by
\begin{equation}
\label{immigration_eq} \tilde\bG(\bu)=\prod_{i=1}^N \tilde
S_i\left(\phi(u_1,\ldots,u_i,\kappa_{i+1}(\bu),\ldots,\kappa_N(\bu))\right).
\end{equation}
\end{theorem}
We tacitly assume that $\kappa_N$ is read as $\eta_N$, the
argument of $\tS_N$ in \eqref{immigration_eq} is $\phi(\bu)$ and
$\bZ^0$ is a given starting distribution. Importantly, an
immediate consequence of \autoref{theorem1} is that we can use
\autoref{thm:branching} to obtain the limiting ({\it
steady-state}) distribution of the joint workload $\bZ^\infty$ at
polling epochs for our polling model, cf. \eqref{eq:br}.
\begin{proof}
Consider the polling system at time $t_n$ and assign the color
$c_i$ to the fluid in $Q_i$ and denote its amount through $x_i$,
for all $i=1,\ldots,N$. Now suppose that fluid
arriving during switch-overs has color $c_0$, and fluid arriving
during  the service of $c_i$-colored fluid has the same color
$c_i$ for all $i=0,\ldots,N$. We stress that our coloring depends
on the color of the fluid in service, i.e., {\em not} on the
queue. Given $x_i$ -- the amount of fluid of color $c_i$ at time
$t_n$, denote  the amount of fluid of color $c_i$ at time
$t_{n+1}$ through $\bR_i^{n+1}(x_i)$ (offspring of type $i$) if
$i\neq 0$, and $\bG^{n+1}$ (immigration) if $i=0$, so that
\begin{equation}
\label{main}
\bZ^{n+1}=\sum_{i=1}^N\bR_i^{n+1}(B^n_i)+\bG^{n+1},\quad
n\ge0.\end{equation}

Obviously the sequences $\{\bR_i^n(x), x\ge0, n\ge 1\}$ and
$\{\bG^n,n\ge1\}$ constitute sequences of i.i.d.\ increasing
stochastic processes and random vectors, respectively. Moreover,
$\bG^n$ is independent from $\bR_i^n$ for $i=1,\ldots,N$, because
the amount of fluid arriving during switch-over periods depends
only on the lengths $S_i$ (independent from anything else) of
those periods and the input process $\bW$, but not on the amount
of fluid already in the system. Denote the common distribution of
$\{\bR_i^n(x), x\ge0, n\ge 1\}$ and $\{\bG^n,n\ge1\}$ by $\bR_i$
and $\bG$, respectively.

Note that at the time instant when the server starts polling
$Q_{i+1}$, the amount of fluid of color $c_i$ is given by
$\bH_i(x_i)$, so
\begin{equation}\label{R_i} \bR_i(x_i)\de\sum_{j=1}^i\be_j
H_{i,j}(x_i)+\sum_{j=i+1}^N\bR_j(H_{i,j}(x_i)),
\end{equation}
where $H_{i,j}(x)$ denotes the $j$th element of $\bH_i(x)$.
Note that the color $c_i$ fluid can appear in the system
only as a replacement of the fluid already present in $Q_i$ at the beginning
of the polling cycle (which corresponds to the part $\sum_{j=1}^i\be_j
H_{i,j}(x_i)$), or as a replacement of the fluid that arrived to the, yet to be served, queues $Q_{i+1},\ldots,Q_N$, during the service of $Q_i$ (which corresponds to the part
$\sum_{j=i+1}^N\bR_j(H_{i,j}(x_i))$).

Backward induction (from $i=N$ to $i=1$) and stationarity and
independence of increments of $\bH_i$ imply the same properties
for $\bR_i$. Hence, $\bR_i$ are L\'evy subordinators, for
$i=1,\ldots,N$. Finally, the mutual independence of $\bR_i^n$ for
$i=1,\ldots,N$ follows from the \autoref{Property2}. Hence
$\{\bZ^n,n\ge0\}$ is a MTJBP.

Next, we compute the Laplace exponent $\kappa_i$ of $\bR_i$. Using
\eqref{R_i} and conditioning on $\bH_i(x)$ we obtain
\begin{align*}
\ee \exp\left(-\bu\cdot\bR_i(x)\right)&=
 \ee\exp\left(-\sum_{j=1}^i u_j H_{i,j}(x)
 -\sum_{j=i+1}^N H_{i,j}(x)\kappa_j(\bu)\right)\\
 &=
 \exp\left(-x\eta_i(u_1,\ldots,u_i,\kappa_{i+1}(\bu),\ldots,\kappa_N(\bu))\right),
\end{align*}
so that~\eqref{recursive_eq} holds.

It is left to compute the LST of $\bG$. First note that we can
write $\bG=\sum_i \bG_i$, where $\bG_i$ are mutually independent
and represent fluid at the end of the polling cycle generated by
the $i$th switch-over. That is,
\[
\bG_i\de\sum_{j=1}^i\be_j W_j(S_i)+\sum_{j=i+1}^N\bR_j(W_j(S_i)).
\]
 Similarly
as above we obtain
\begin{align*}\ee \exp(-\bu\cdot\bG_i)&=
 \ee\exp\left( -\sum_{j=1}^i u_j W_j(S_i)
 -\sum_{j=i+1}^N W_j(S_i)\kappa_j(\bu)\right)\\
 &=
 \ee\exp\left(-S_i\phi(u_1,\ldots,u_i,\kappa_{i+1}(\bu),\ldots,\kappa_N(\bu))\right),\end{align*}
 which proves~\eqref{immigration_eq}.
\end{proof}

Let $\bZ_i$ and $\bE_i$ denote the random variable having the steady-state distribution of the joint amount of fluid
in each queue at the beginning of a visit (polling instant) to
$Q_i$ and at the end of a visit (switching instant) to
$Q_i$, respectively.
\begin{samepage}
\begin{cor} \label{Cor}
$\bZ_i$ and $\bE_i$ can be related to each other by
\begin{equation}
\label{Z_i}
 \tilde\bZ_{i+1}(\bu)=\tilde\bE_i(\bu)\tS_i(\phi(\bu))
\end{equation}
 and
 \[
\tilde\bE_i(\bu)=\tilde\bZ_i(u_1,\ldots,u_{i-1},\eta_i(\bu),u_{i+1},\ldots,u_N),
 \]
where
\[
\tilde\bZ_1(\bu)=\prod_{k=0}^\infty\tilde\bG(\bvarphi^{(k)}(\bu)),
\]
with $\tilde\bG$ and $\bvarphi$ given by \autoref{theorem1}.
\end{cor}
\end{samepage}
\begin{rem}
By the same token, we can, for arbitrary $i=1,\ldots,N$,  find
$\tilde\bZ_i(\bu)$ as well (renumber such that $Q_i$ becomes
$Q_1$). The infinite product formula for the Laplace transform
of the distribution of $\bZ_1$ is typical in the area of (classical)
polling models. This kind of formula can be numerically inverted
to obtain various performance metrics, see \citet{Abate95, Abate06} and \citet{Choudhury96}. Such an infinite product already arises in the case of a M/G/1 queue ($N=1$) with gated vacations. For example \cite[Section 2.5, Formula (5.19)]{Takagi2}  gives the following relation for $Q(z)$, the generating function of the queue length distribution at the end of a vacation:
\[
Q(z)=Q(B^*(\lambda(1-z))V^*(\lambda(1-z)),
\]
where $\lambda$ denotes arrival rate and $B^*(\cdot)$, $V^*(\cdot)$ are the LST of service time and vacation length, respectively. This immediately results in
\[
Q(z)=\prod_{i=0}^\infty V^*(\lambda(1-\delta^{(i)}(z))),
\]
where $\delta^{(i)}(z)=B^*(\lambda(1-\delta^{(i-1)}(z)))$, $i=1,2,\ldots,$ $\delta^{(0)}(z)=z$.
In the special case of a single queue ($N=1$) and exhaustive service, the infinite product degenerates to $V^*(\lambda(1-z))$ because at the end of each visit, the system has become empty.

For arbitrary $N$ and classical Poisson input processes, \citet{Resing} presents the joint queue length generating function at epochs the server begins a visit to $Q_1$. This generating function is also given in the form of an infinite product, cf. \cite[Theorem 1 and Theorem 3]{Resing}.
\end{rem}
\begin{rem}
The interpretation of the above infinite-product expression for
$\tilde\bZ_1(\bu)$ is the following. The terms of the infinite
product correspond to independent contributions to the workload
vector at a polling instant of $Q_1$. The $0$th term represents
work still present, that has arrived during the switch-over
periods in the cycle that has just ended. The 1st term represents
work that has arrived during the service of work that had arrived
one cycle earlier during switch-over periods. And so on: the $k$th
term, $k=1,2,\dots$, represents work that was initiated $k$ cycles
before the last cycle by ancestral work arriving during a
switch-over period.
\end{rem}
\begin{rem}
A special case of our model is the fluid polling model studied in
\citet{Czerniak}: there the L\'evy input
reduces to $N$ linear deterministic processes. Another special
case of our model is the classical polling model in which the
L\'evy input corresponds to $N$ independent CPPs.
\end{rem}
\begin{rem}
\label{rem:GG} From the proof of \autoref{theorem1} we can
clearly see the importance of the branching property. Most
importantly, it implied the distributional equality~\eqref{R_i}
for the distribution of $\bR_i$. Then the distribution of $\bG$
follows in terms of the Laplace exponents of the $\bR_i$'s.
However, to establish the MTJBP structure of the L\'evy-driven
polling system, all we need to show is that \eqref{main}
(or equivalently~\eqref{def:mult}) holds with independent components.

One can think of disciplines that do not satisfy
\autoref{Property2}, but for which still the workload
evolution can be described by \eqref{main}. An example is
the {\it globally-gated} discipline (see \cite{GG}), which works
as follows. At the beginning of each cycle, all fluid in
$Q_1,\ldots,Q_N$ is marked. During the next cycle, the server
serves all the marked fluid. The newly arrived fluid, however, has
to wait until being marked at the next cycle-beginning, and will
be served during the next cycle. This discipline does not satisfy
\autoref{Property2}, but assuming that the server works with
rate 1, an equation of the form~\eqref{main} can be derived with
$\bR_i^n\de \bW$ and $\bG=\sum\bG_i$, where $\bG_i\de \bW(S_i)$.
Therefore such a model also has the MTJBP structure with
branching mechanism $\boldsymbol\kappa(\bu)=(\phi(\bu),\ldots,\phi(\bu))$ and 
immigration LST $\tilde\bG(\bu)=\prod \tilde S_i(\phi(\bu))$.
\end{rem}
As in the introduction, it should be noted that a relation between L\'evy-driven polling systems 
and continuous-state space branching processes was
considered before by \citet{altman}. 
In that work, the assumption is imposed that all queues are fed by 
{\it identical} L\'evy subordinators, and it is precisely this assumption 
that enables the construction of a continuous state-space branching process. Moreover,
the relation in \cite{altman} is used only to derive the first two
waiting-time moments and not to determine the structure of the branching process itself.

\section{Steady-state distribution at an arbitrary epoch}
\label{Steady-state distribution}

Having determined the LST of the joint steady-state workload at
polling and switching instants in the previous section, we now
concentrate on its counterpart {\em at an arbitrary instant in
time}. It should be noted that this distribution was not even
found before for the classical polling models with independent CPP
inputs, except for the marginal distributions. In order to do so,
we need to make the notion of service disciplines more precise.

Firstly, recall that work in the system (while the server is at
$Q_i$) evolves according to a L\'{e}vy process
\[
\bA_i(t)
:=\Big(W_1(t),\ldots,W_{i-1}(t),A_i(t),W_{i+1}(t),\ldots,W_{N}(t)\Big),
\]
where $A_i(t)$ can be any spectrally positive L\'{e}vy process
with negative drift. The Laplace exponent of $\bA_i(t)$ is denoted
through $\phi^{\bA}_i(\bu)$. Let $\mathcal F_i$ be an augmented
right-continuous filtration, such that $\bA_i(t)$ is a L\'{e}vy
process with respect to $\mathcal F_i$ (one can take an augmented
natural filtration of $\bA_i(t),t\geq 0$). Let $\tau_i(x)$ denote
the time the server spends at $Q_i$ when it finds $x$ units of
work in this queue upon arrival. Loosely speaking, $\tau_i(x)$ is
a stopping rule for the server, which observes the process
$\bA_i(t)$. This motivates the following assumption.
\begin{assumption}
\label{asm_stoppingTime}
$\tau_i(x)$ is an $\mathcal F_i$-stopping time for every $x$.
\end{assumption}
\autoref{asm_stoppingTime} ensures that the disciplines are
`non-anticipating'. For example, we exclude the cases, when 
the server decides to stop the service if it `sees' that, for instance, the cumulative input
in the next $T$ units of time is less than some $\varepsilon$.
The above assumption expresses the fact, that a service strategy or discipline
should be based only on the knowledge of the evolution of the system up to 
the current time-point.
Observe that for the gated
discipline $\tau_i(x)=x$ and for exhaustive
$\tau_i(x)=T_i(x)=\inf\{t\geq 0:A_i(t)=-x\}$ (recalling the
definition of $T_i(x)$ from \autoref{Model description}). It
is easily verified that  both are $\mathcal F_i$-stopping times,
as desired.

We also require that a discipline is `work conserving', that is,
the server never stays at a queue after it became empty. This is
made precise in the following assumption.
\begin{assumption}\label{asm:workConservation}
For every $x$ it holds that $\tau_i(x)\leq T_i(x)$ a.s.
\end{assumption}
Note that the gated discipline ($A_i(t)=W_i(t)-t$) and the
exhaustive discipline are always work conserving. Because of
\autoref{asm:workConservation} one does not need to
consider reflection of the workload process at 0, hence the
workload replacement $\bH_i(x)$ is given by
\[\bH_i(x)=x\be_i+\bA_i(\tau_i(x)).\]
Moreover, $\ee\tau_i(x)\leq \ee T_i(x)<\infty$, because $A_i(t)$
has negative drift. Hence using Wald's identity twice to L\'evy
processes $\bH_i$ and $\bA_i$ we get
\begin{equation}\label{tau_linear}\ee \tau_i(x)=x\ee\tau_i(1).\end{equation}

The following result presents a Kella-Whitt type martingale, which
is a key tool in deriving the workload LST at an arbitrary time.
\begin{prop}
\label{Martingale}
\[
M_i(t):=e^{-\bu\cdot\bA_i (t)}-1+\phi^{\bA}_i(\bu)
\int_0^te^{-\bu\cdot\bA_i(s)}\, ds,\quad t\ge 0,
\]
is a zero mean martingale with respect to filtration $\mathcal
F_i$.
\end{prop}
\begin{proof}
Apply \citet[Theorem 1]{Kella} to the
one-dimensional L\'evy process $\bu\cdot\bA_i$ (with respect to
filtration $\mathcal F_i$).
\end{proof}
Applying Doob's Optional Sampling theorem to the martingale $M_i$
from \autoref{Martingale} and stopping time
$\tau_i(x)\wedge n$ yields
\[
\ee\int_0^{\tau_i(x)\wedge n} e^{-\bu\cdot\bA_i(s)}\,ds
=\frac{1}{\phi^{\bA}_i(\bu)}\ee\left(1-e^{-\bu\cdot\bA_i(\tau_i(x)\wedge
n)}\right).
\]
Taking $n\rightarrow\infty$ and applying the monotone convergence
theorem on the left and the dominated convergence theorem on the
right ($e^{-\bu\cdot\bA_i(\tau_i(x)\wedge n)}\leq e^{u_ix}$) we
obtain
\begin{equation}
\label{Doob}
 \ee\int_0^{\tau_i(x)} e^{-\bu\cdot\bA_i(s)}\,ds
 =\frac{1}{\phi^{\bA}_i(\bu)}\ee\left(1-e^{-\bu\cdot\bA_i(\tau_i(x))}\right).
\end{equation}
We are now ready to state the second main result of this paper.
\begin{samepage}
\begin{theorem}
\label{theorem2} Consider the polling system described in 
\autoref{Model description} and suppose that the disciplines at every
queue satisfy \autoref{asm_stoppingTime},
\autoref{asm:workConservation} and 
\autoref{Property2}. Then the LST of the steady-state distribution of
the joint amount of fluid $\bF$ at an arbitrary epoch is given by
\begin{equation}\label{wynik} \ee e^{-\bu\cdot\bF}=
\frac{N(\bu)}{\displaystyle \sum_{i=1}^N \ee S_i+
\sum_{i=1}^N\ee\tau_i(1)\ee B_{i,i}},
\end{equation}
where
\begin{equation}
\label{numerator} N(\bu)=\sum_{i=1}^N \left(
\frac{\tilde\bZ_i(\bu)-\tilde\bE_i(\bu)}{\phi^{\bA}_i(\bu)}+
\frac{\tilde\bE_i(\bu)-\tilde\bZ_{i+1}(\bu)}{\phi(\bu)}\right)
\end{equation}
and $\tilde\bZ_i$, $\tilde\bE_i$ are as given in 
\autoref{Cor}.
\end{theorem}
\end{samepage}
\begin{proof} Let $\bF(t)$ be the amount of fluid in each queue at time
$t$ within a cycle $C$, assuming that we start in stationarity.
The LST of $\bF$ is calculated by dividing the expected area of
the function $e^{-\bu\cdot\bF(t)}$ over the cycle $C$, by the
expected cycle time $\ee C$:
\[{\displaystyle
\tilde\bF(u)=\frac{\displaystyle \ee\int_0^C
e^{-\bu\cdot\bF(t)}\,dt}{\ee C}.}
\]
Note that
\[
C\de\sum_{i=1}^N \left(S_i+\tau_i(B_{i,i})\right),
\]
from which, in combination with~\eqref{tau_linear}, the
denominator in \eqref{wynik} follows.

An arbitrary cycle of length $C$ can be divided into intervals
corresponding to visit periods $V_i$ to $Q_i$ and switching (idle)
periods $I_i$ between $Q_i$ and $Q_{i+1}$. Conditioning on $S_i$
and using Fubini's theorem, we find
\begin{align*}
\cs_i(\bu)&:=\ee\int_{I_i}e^{-\bu\cdot\bF(t)}\,dt
=\ee\int_0^{S_i}e^{-\bu\cdot(\bE_i+\bW(t))}\,dt
=\tilde\bE_i(\bu)\frac{1-\tS_i(\phi(\bu))}{\phi(\bu)}\\
&=\frac{\tilde\bE_i(\bu)-\tilde\bZ_{i+1}(\bu)}{\phi(\bu)}.
\end{align*}
On the intervals $V_i$ we have
\[
\cl_i(\bu):= \ee\int_{V_i}e^{-\bu\cdot\bF(t)}\,dt
=\ee\int_0^{\tau_i(B_{i,i})}e^{-\bu\cdot(\bZ_i +
\bA_i(t))}\,dt.
\]
Conditioning on $\bZ_i$ and applying \eqref{Doob} yields
\begin{align*}
\cl_i(\bu)& =\frac{1}{\phi^{\bA}_i(\bu)}\left(\ee
e^{-\bu\cdot\bZ_i}-\ee e^{-\bu\cdot(\bZ_i +
\bA_i(\tau_i(B_{i,i}))}\right) \\
&=\frac{1}{\phi^{\bA}_i(\bu)}\left(\ee e^{-\bu\cdot\bZ_i}-\ee
e^{-\bu\cdot\bE_i}\right)=\frac{\tilde\bZ_i(\bu)-\tilde\bE_i(\bu)}
{\phi^{\bA}_i(\bu)}.
\end{align*}
Finally,
\[
\ee\int_0^C e^{-\bu\cdot\bF(t)}\,dt
=\sum_{i=1}^N\left(\cl_i(\bu)+\cs_i(\bu)\right).
\]
\end{proof}
Note that all the quantities appearing in the statement of
\autoref{theorem2} are computable, in the sense that they
follow from the results presented in \autoref{Polling and
Jirina}: the transforms $\tilde\bZ_i(\bu)$ and $\tilde \bE_i(u)$
are given in \autoref{Cor} and $\ee
B_{i,i}=-\partial\tilde\bZ_i/\partial u_i(\0)$.

\section{Varying input processes}
\label{VIP}

In \autoref{Model description} and \autoref{Polling and Jirina} we
considered a polling model, with fixed input $\bW$ characterized
by its Laplace exponent $\phi$. However, to derive our results,
all we needed was the knowledge of $\phi$ and $\phi^{\bA}_i$ for
$i=1,\ldots, N$ as remarked at the end of \autoref{Model
description} and not the fact that $\phi_i^{\bA}$ are related to
each other or that $\phi$ is fixed. That is why we can allow our
input to change between embedded epochs.

More precisely, let $\bW_i$ and $\hat\bW_i$ be sequences of
$N$-dimensional subordinators for $i=1,\ldots,N$. When the server
arrives at $Q_i$ then the input process changes to $\bW_i$, and
when the server leaves $Q_i$ the input process switches to
$\hat\bW_i$. Let us denote the Laplace exponents of these
processes by $\phi_i$ and $\hat\phi_i$, respectively. The process
$\bA_i(t)$ becomes
\[
\bA_i(t)
:=\Big(W_{i,1}(t),\ldots,W_{i,i-1}(t),A_i(t),W_{i,i+1}(t),\ldots,W_{i,N}(t)\Big),
\] where $A_i(t)$ is an arbitrary spectrally positive L\'{e}vy process with negative drift modelling the server's work at $Q_i$.
Again, denote the Laplace exponent of $\bA_i$ by $\phi^{\bA}_i$.

We still consider disciplines satisfying \autoref{Property2},
so that for the gated discipline \eqref{gated} changes to
\[
\eta_i(\bu)=\phi_i(\bu),
\]
and for the exhaustive discipline \eqref{ex} changes to
\[
\eta_i(\bu)=\psi_i(\bu),
\]
with $\psi_i$ such that $\phi_{i,\bu}(\psi_i(\bu))=0$ for
$\phi_{i,u}(\theta)=\phi_i^{\bA}(u_1,\ldots,
u_{i-1},\theta,u_{i+1},\ldots,u_N)$, where $\theta\geq 0$. The
resulting process $\{\bZ^n,n\ge 1\}$ of the joint amount of the
fluid in the different queues at time points $t_n$ also
constitutes an MTJBP with branching mechanism $\bvarphi$ given by
\eqref{recursive_eq} and immigration LST given by
\begin{equation}
\label{immigration2}
\tilde\bG(\bu)=\prod_{i=1}^N\tS_i(\hat\phi_i(u_1,\ldots,u_i,\kappa_{i+1}(\bu),\ldots,\kappa_N(\bu)))
\end{equation}
instead of \eqref{immigration_eq}. The argument given in the proof
of \autoref{theorem1} stays valid. Then \eqref{Z_i} in \autoref{Cor} changes into
\begin{equation}
\label{nZ_i}
\tilde\bZ_{i+1}(\bu)=\tilde\bE_i(\bu)\tS_i(\hat\phi_i(\bu)).
\end{equation}
The statement of \autoref{theorem2} also still holds with
\[
N(\bu)=\sum_{i=1}^N \left(\frac{ \tilde\bZ_i(\bu)-\tilde\bE_i(\bu)
}{\phi^{\bA}_i(\bu)}+\frac{\tilde\bE_i(\bu)-\tilde\bZ_{i+1}(\bu)}{\hat\phi_i(\bu)}
\right).\]
\begin{rem}
For independent compound Poisson input processes, the case of
varying input has been studied in \cite{BoxmaB}. There it is
assumed that the arrival process at $Q_i$, when the server is at
$Q_j$, is a Poisson process with rate $\lambda_{ij}$. Under the
assumption of branching-type service disciplines,  the joint
queue-length distribution at polling instants is derived in
\cite{BoxmaB}.
\end{rem}

\section{Ergodicity}
\label{Ergodicity}

Consider the polling model with varying input, as was presented in \autoref{VIP}. 
The stability condition of such a system is
given in terms of the Perron-Frobenius eigenvalue $\rho_M$ of the
mean matrix $M$ associated to a certain branching process, see
\autoref{thm:branching}. This branching process in turn is
specified in \autoref{theorem1} (or, more precisely, in its
generalization to varying input, as presented in 
\autoref{VIP}) in terms of $\eta_i$ among other quantities. This leads
to a rather non-transparent stability criterion, as opposed to the
`$\rho<1$' type of criteria one usually encounters in queueing
theory. In addition, it is not clear how the criterion  depends on
the disciplines used at different queues. The goal of this section
is to make the stability condition more explicit, and to show that
it is, under quite general circumstances, {\it not} affected by
the service discipline.

In this section we assume that the disciplines at all queues
satisfy \autoref{asm_stoppingTime} and
\autoref{asm:workConservation}. Moreover, we exclude the
degenerate case when $\tau_i(x)=0$ (never serving $Q_i$). In this
setting the stability condition
 can be greatly simplified. Importantly,
 we show that it can be expressed in terms of properties of the {\it rate matrix} $A=(a_{ij})$
 rather than properties of the {\it mean matrix} $M$; here $a_{ij}=\ee A_{i,j}(1)$, that is,
$a_{ij}$ is the rate of the work evolution at $Q_j$ while the
server is at $Q_i$. Note that $A$ has non-negative off-diagonal
elements, hence by Perron-Frobenius theory it has a real
eigenvalue $\rho_A$ which is larger than the real part of any
other eigenvalue of~$A$.

\begin{lem}\label{lem:erg}Let $A$ be irreducible. Then it holds that
\begin{itemize}
  \item if $\rho_A<0$ then $\rho_M<1$ (subcritical);
  \item if $\rho_A=0$ then $\rho_M=1$ (critical);
  \item if $\rho_A>0$ then $\rho_M>1$ (supercritical).
\end{itemize}
In the supercritical case there exists a positive vector $\bw$
such that $M\bw>\bw$.
\end{lem}
\begin{proof}
Let us consider the polling model from \autoref{VIP} and
denote the MTJBP associated to it by $(\hat{\bZ^n})$. This MTJBP
has a corresponding branching mechanism $\bvarphi$ and immigration
$\bG$. Let us construct a new MTJBP $(\bZ^n)$ with the same
branching mechanism $\bvarphi$ but without immigration, i.e., with
$\bG$ set to 0. Such an MTJBP obviously corresponds to a polling
model with the same characteristics as the starting model, but
with switch-over times set to 0. Moreover the mean matrix $\hat M$
is the same as the mean matrix $M$.

From the definition of $M$ its $i$th row is given by $\bm_i=\ee
\left(\bZ^1|\bZ^0=\be_i\right)$. In the following we assume without
loss of generality,
that $\bZ^0=\be_i$. We can write
\[
\bZ^1=\be_i+\sum_{k=i}^N\bA_k\left(\tau_k(F_k)\right),
\]
where $F_k$ denotes the fluid in $Q_k$ upon the server's arrival
to this queue (in the first cycle). Using Wald's identity and the
linearity of $\ee\tau_k(x)$, as given in \eqref{tau_linear}, we
obtain
\[\bm_i=\ee\bZ^1=\be_i+\sum_{k=i}^N\ee\tau_k(1)\ee F_k\ba_k,\]
where $\ba_k$ is the $k$th row of $A$. Let $\bw>\0$ be an
eigenvector of $A$ with positive elements associated to $\rho_A$,
which exists by the Perron-Frobenius theory. Then
\[
\bm_i\cdot\bw=w_i+\rho_A\sum_{k=i}^N\ee\tau_k(1)\ee F_k w_k.
\]
Note that $\ee\tau_k(1)>0$ and $\ee F_k\geq 0$ for all $k$, and
$\ee F_i=1$, because $\bX^0=\be_i$ and $S_k=0$ for all $k$.
Therefore, according to $\rho_A<0,\rho_A=0$, and $\rho_A>0$ we
obtain $M\bw<\bw,M\bw=\bw$, and $M\bw>\bw$, respectively. Now the
claim follows from \autoref{subinvariance} in the Appendix.
\end{proof}

In the following we assume that the total work arriving during the
switch-overs in one polling cycle is not identically 0, that is we
can not erase the switch-over periods without changing the system.
The next corollary is an immediate consequence of
\autoref{lem:erg} and \autoref{thm:branching}, in
conjunction with the comments following the proof of
\autoref{thm:branching}.
\begin{cor}\label{cor:erg}
Let $A$ be irreducible. Then it holds that
\begin{itemize}
  \item if $\rho_A<0$, then the polling system is stable;
  \item if $\rho_A>0$, then the polling system is unstable.
\end{itemize}
\end{cor}
Note that the stability of our polling system depends only on the
input and does not depend on the disciplines used at different
queues. Clearly, this result strongly relies on the fact that the
disciplines are work conserving, see
\autoref{asm:workConservation}, and satisfy
\autoref{Property2}.

Finally we make a comment on a simplified model from
\autoref{Model description}, where the input does not depend
on the location of the server, and in which the server works at
unit speed. That is $A_i=W_i(t)-t$ for a fixed $\bW$. Denote the
mean rate of the input into $Q_i$ by $\rho_i>0$ and the mean total
rate by $\rho=\sum_{i=1}^N \rho_i$. This means that
$a_{ij}=\rho_j-\delta_{ij}$, so that $A$ is irreducible and
$A\1=(\rho-1)\1$. Apply \autoref{subinvariance} in the Appendix
and \autoref{cor:erg} to see that we obtain the expected
stability condition: the system is stable if $\rho<1$ and is
unstable if $\rho>1$.

\section{Discussion and concluding remarks}
\label{DACR}

In this paper we analyzed a general class of L\'evy driven polling
models. Exploiting the relation with multi-type Ji\v{r}ina
processes, we determined the LST of the joint stationary workload
distribution. The collection of results presented in this paper is
rich, as they cover various known results as special cases, but
also constitute a broad range of new results, even for the class
of classical polling models with compound Poisson input.

Various extensions and ramifications can be thought of. For
instance, where the present paper considers cyclic polling (i.e.,
in every cycle all $N$ queues are served in a cyclic manner), any
system with a fixed polling table can be dealt with similarly. The
class of models in which the polling order is {\it random},
however, does {\it not} fit in the class of MTJBPs, as was already
observed for MTBP in \cite{Resing}, and can therefore not be
analyzed by our methods. Another direction for future research
relates to the use of the LST of the workload distribution to
obtain insight into the corresponding tail probabilities, as was
done for a specific polling model in \cite{Deng}.

It would also be interesting to investigate whether a work
decomposition property for the polling model with/without
switch-over times exists, cf. \citet{BoxmaA}. Yet another
direction for future research is to derive the joint steady-state
queue length distribution {\it at an arbitrary epoch} for the
classical polling model with CPP input, using a martingale
technique as we have employed for the workload in
\autoref{Steady-state distribution}.

We envisage that a variant of our approach can deal with an even
larger class of service disciplines than the class of branching
disciplines. As was noted in \autoref{rem:GG} of
\autoref{Polling and Jirina}, the so-called {\it
globally-gated} service discipline does not qualify as branching
type, but we showed that it is nevertheless possible to find a
corresponding MTJBP. The ideas presented in \autoref{rem:GG}
indicate under what circumstances still MTJBPs can be
constructed.

\appendix
\section{}\label{appendix1}
The following Lemma is a variation of~\cite[Theorem 1.6]{Seneta3}.
\begin{lem}\label{subinvariance}
Let $M$ be a square matrix with non-negative off-diagonal
elements. Let $\rho_M$ be its Perron-Frobenius eigenvalue (the
eigenvalue with maximal real part) and $\bw$ be a positive vector.
Then
\begin{align*}
M\bw<\bw\text{ implies }\rho_M<1,\\
M\bw=\bw\text{ implies }\rho_M=1,\\
M\bw>\bw\text{ implies }\rho_M>1.
\end{align*}
\end{lem}
\begin{proof}
Let $\bv\geq \0$ be a left eigenvector (i.e., a row vector) of $M$
corresponding to $\rho_M$. If $M\bw<\bw$, then $\rho_M(\bv \cdot
\bw)=\bv M\bw<\bv\cdot \bw$. Hence $\rho_M<1$. The other two
statements follow similarly.
\end{proof}

\section*{Acknowledgments} The authors thank E.\ Altman (INRIA, France),
S.\ Meyn (University of Illinois at Urbana-Champaign, United
States), and J.A.C.\ Resing (Eindhoven University of Technology,
the Netherlands) for useful discussions.

\small
\bibliography{polling}
\end{document}